\documentclass[12pt]{amsart}
\usepackage{amssymb}

\newtheorem*{THM}{Main Result}
\newtheorem*{ThM}{Theorem B}
\newtheorem{lemma}{Lemma}
\newtheorem{cor}{Corollary}
\newtheorem{claim}{Subclaim}
\newtheorem*{defi}{Definition}

\newcommand\name[1]{\dot{#1}}
\def\forces{\Vdash}
\def\s{\subseteq}

\def\br{\blacktriangleright}

\def\sd{\framebox[3.6mm][l]{$\diamondsuit$}\hspace{0.5mm}{}}

\DeclareMathOperator{\dom}{dom}
\DeclareMathOperator{\otp}{otp}
\DeclareMathOperator{\acc}{acc}
\DeclareMathOperator{\chr}{Chr}
\DeclareMathOperator{\rng}{rng}
\DeclareMathOperator{\zfc}{\hbox{\textsf{ZFC}}}
\DeclareMathOperator{\gch}{\hbox{\textsf{GCH}}}

\begin{document}
\author{Assaf Rinot}
\address{Department of Mathematics, Bar-Ilan University, Ramat-Gan 52900, Israel.}
\urladdr{http://www.assafrinot.com}
\title{Hedetniemi's conjecture for uncountable graphs}
\begin{abstract}It is proved that in G\"odel's constructible universe,
for every successor cardinal $\kappa$, there exist graphs $\mathcal G$ and $\mathcal H$
of size and chromatic number $\kappa$, for which the tensor product graph $\mathcal G\times\mathcal H$ is countably chromatic.
\end{abstract}

\maketitle

\section{Background}A graph $\mathcal G$ is a pair $(G,E)$, where $E\s[G]^2=\{ \{x,y\} \mid x,y\in G\ \&\  x\neq y\}$.
The chromatic number of $\mathcal G$, denoted $\chr(\mathcal G)$, is the least cardinal $\kappa$ such that $G$ is the union of $\kappa$ many $E$-independent sets.
Equivalently, $\chr(\mathcal G)$ is the least cardinal $\kappa$, for which there exists an $E$-chromatic $\kappa$-coloring of $G$,
that is, a coloring $\chi:G\rightarrow\kappa$ that satisfies $\chi(x)\neq\chi(y)$
whenever $xEy$.

Given graphs $\mathcal G_0=(G_0,E_0)$ and $\mathcal G_1=(G_1,E_1)$, the tensor product graph $\mathcal G_0\times\mathcal G_0$ is defined as $(G_0\times G_1,E_0*E_1)$, where:
\begin{itemize}
\item $G_0\times G_1:=\{ (g_0,g_1)\mid g\in G_0, h\in G_1\}$;
\item $E_0*E_1:=\{ \{(g_0,g_1), (g_0',g_1')\}\mid (g_0,g_0')\in E_0 \text{ and }(g_1,g_1')\in E_1\}$.
\end{itemize}

Clearly, $\chr(\mathcal G_0\times\mathcal G_1)\le\min\{\chr(\mathcal G_0),\chr(\mathcal G_1)\}$.
Hedetniemi's conjecture \cite{MR2615860} states that
$\chr(\mathcal G_0\times\mathcal G_1)=\min\{\chr(\mathcal G_0),\chr(\mathcal G_1)\}$.

In a paper from 1985, Hajnal \cite{MR815579} proved that the above equality may fail in the case that $\mathcal G_0,\mathcal G_0$ are infinite graphs.
Specifically, he proved that for every infinite cardinal $\lambda$,
there exist graphs $\mathcal G_0,\mathcal G_1$ of size and chromatic number $\lambda^+$,
such that $\chr(\mathcal G_0\times \mathcal G_1)=\lambda$.\footnote{Todorcevic pointed out
that some instances of the failure of the infinite Hedetnimei conjecture were implictly available in the early 1970's \cite{MR0419229}.
See also \cite{MR657114} and \cite{MR1667145}.} Subsequently, Soukup \cite{MR937544} established the consistency
of $\zfc+\gch$ together with the existence of graphs $\mathcal G_0,\mathcal G_1$ of size and chromatic number $\aleph_2$
such that $\chr(\mathcal G_0\times \mathcal G_1)=\aleph_0$.
So, in Hajnal's example, we have
$\chr(\mathcal G_0\times \mathcal G_1)=\log(\min\{\chr(\mathcal G_0),\chr(\mathcal G_1)\})$,
and in Soukup's example, we have
$\chr(\mathcal G_0\times \mathcal G_1)=\log\log(\min\{\chr(\mathcal G_0),\chr(\mathcal G_1)\})$.
This raises the general question of how badly may Hedetniemi's conjecture fail in the infinite case. Specifically:
\begin{enumerate}
\item (Hajnal, \cite{pims2004}; Soukup, \cite{MR2432534}) Is it consistent with $\zfc+\gch$ that there are graphs $\mathcal G_0,\mathcal G_1$
with $\chr(\mathcal G_0)=\chr(\mathcal G_1)=\aleph_3$ and $\chr(\mathcal G_0\times\mathcal G_1)=\aleph_0$?
\item (Hajnal, \cite{MR815579}) Is it consistent with $\zfc+\gch$ that there are graphs $\mathcal G_0,\mathcal G_1$
with $\chr(\mathcal G_0)=\chr(\mathcal G_1)\ge\aleph_\omega$ and $\chr(\mathcal G_0\times\mathcal G_1)<\aleph_\omega$?
\end{enumerate}

Hajnal noticed the following obstruction. He proved that if $\mathcal G_0,\mathcal G_1$
are pair of graphs of size and chromatic number $\kappa$ whose tensor product is countably chromatic,
then for every $i<2$, $\mathcal G_i$ is $(\aleph_0,\kappa)$-chromatic,
that is, $\chr(\mathcal G_i)=\kappa$, but every subgraph of $\mathcal G_i$ of size $<\kappa$ has chromatic number $\le\aleph_0$.

A model of $\zfc+\gch$ in which there exists an $(\aleph_0,\aleph_2)$-chromatic graph of size $\aleph_2$ was obtained by Baumgartner \cite{MR736618}
via a very complicated notion of forcing. Then, Soukup's model \cite{MR937544} of $\zfc+\gch$ with graphs $\mathcal G_0,\mathcal G_1$ of size and chromatic number $\aleph_2$
such that $\chr(\mathcal G_0\times \mathcal G_1)=\aleph_0$ is a further sophistication of Baumgartner's forcing.
As Baumgartner's forcing does not seem to generalize to yield a model of an  $(\aleph_0,\aleph_3)$-chromatic graph,
an answer to the above questions would require an alternative construction.
Alternative constructions were soon offered by Komj{\'a}th \cite{MR941243}, Soukup \cite{MR1066404} and Shelah  \cite{sh:347},
with the latter having the feature of $\gch$ holding.
However -- as Hajnal mentioned in his survey paper \cite{pims2004} --
there was no success in tailoring these constructions.

Recently, the author \cite{rinot12} found yet another construction that builds on the concept of \emph{Ostaszewski square} that was introduced in \cite{rinot11}.
He denoted these graphs by $G(\overrightarrow C)$ and isolated the features of $G$ and $\overrightarrow C$
that make $G(\overrightarrow C)$ into $(\aleph_0,\lambda^+)$-chromatic graphs.
Even more recently \cite{rinot15}, he proved that the chromatic number of these graphs may be turned countable
via a $(\le\lambda)$-distributive notion of forcing. In this paper,
these new findings are combined together with the basic idea of Hajnal's 1985 construction
to obtain the ultimate counterexample to the infinite Hedetniemi conjecture.

\begin{THM} If $\lambda$ is an uncountable cardinal, and $\sd_\lambda$ holds,
then there exist graphs $\mathcal G_0,\mathcal G_1$ of size $\lambda^+$ such that:
\begin{itemize}
\item $\chr(\mathcal G_0)=\chr(\mathcal G_1)=\lambda^+$;
\item $\chr(\mathcal G_0\times\mathcal G_1)=\aleph_0$.
\end{itemize}
\end{THM}
\begin{cor}In G\"odel's constructible universe,
for every successor cardinal $\kappa$, there exist graphs $\mathcal G_0$ and $\mathcal G_1$
of size and chromatic number $\kappa$, for which the tensor product graph $\mathcal G_0\times\mathcal G_1$ is countably chromatic.
\end{cor}
\begin{cor} In G\"odel's constructible universe, for every positive integer $n$,
there exist infinite graphs $\mathcal G_0$ and $\mathcal G_1$ such that
$$\chr(\mathcal G_0\times\mathcal G_1)=\underbrace{\log\cdots\log}_{n\text{ times}}(\min\{\chr(\mathcal G_0),\chr(\mathcal G_1)\}).$$
\end{cor}
In Section \ref{sec3} below, we also settle the generalized problem concerning the tensor product of $n+1$ many graphs ($0<n<\omega$).

\section{Proof of the main result}
Suppose that $\lambda$ is an uncountable cardinal, and $\sd_\lambda$ holds.
By $\sd_\lambda$ and a standard partitioning argument \cite{MR0491194},\cite{rinot11}, we may then fix a sequence $\langle (D_\alpha,X_\alpha)\mid \alpha<\lambda^+\rangle$,
along with a function $h:\lambda^+\rightarrow\lambda^+$ such that all of the following hold:
\begin{enumerate}
\item  for every limit $\alpha<\lambda^+$, $D_\alpha$ is a club in $\alpha$ of order-type $\le\lambda$;
\item if $\beta\in\acc(D_\alpha)$, then $D_\beta=D_\alpha\cap\beta$, $X_\beta=X_\alpha\cap\beta$ and $h(\beta)=h(\alpha)$;\footnote{Here, $\acc(A):=\{ \alpha\in\sup(A)\mid \sup(A\cap\alpha)=\alpha>0\}.$}
\item for every subset  $X\s\lambda^+$, a club $E\s\lambda^+$ and $\varsigma<\lambda^+$,
there exists a limit $\alpha<\lambda^+$ with $\otp(D_\alpha)=\lambda$,
such that $h(\alpha)=\varsigma$, $X_\alpha=X\cap\alpha$ and $\acc(D_\alpha)\s E$.
\end{enumerate}

Clearly, $\langle X_\alpha\mid \alpha\in G_i\rangle$ is a $\diamondsuit(G_i)$-sequence,
where  $G_i:=\{\alpha<\lambda^+\mid h(\alpha)=i\ \&\ \otp(D_\alpha)=\lambda\}$.
Since, $G_0$ and $G_1$ are nonreflecting and pairwise-disjoint stationary sets,
it is then natural to use $G_0(\overrightarrow D)$ and $G_1(\overrightarrow D)$ as
the building blocks of our graphs.\footnote{The graph $G(\overrightarrow D)$ was introduced in \cite{rinot12},
and it was proven there that if $\overrightarrow D$ is a $\square_\lambda$-sequence, and $G$ is a nonreflecting subset of $\lambda^+$,
then any subgraph of $G(\overrightarrow D)$ of size $<\lambda^+$ is countably chromatic.}
Loosely speaking, one of the features that we would need is the ability to kill (via forcing) the guessing feature of $\langle X_\alpha\mid \alpha\in G_0\rangle$,
while preserving the features  of  $\langle X_\alpha\mid \alpha\in G_1\rangle$, and vice versa.
For this, we shall borrow an idea from the proof of \cite[Theorem 2.4]{sh:64}, where a model of $\diamondsuit(\omega_1\setminus S)+\neg\diamondsuit(S)$ was obtained for the first time.\footnote{
The proof is not given in \cite{sh:64}, rather, it is given as the proof of Theorem 2.4 from \cite{sh:98}.
Personally, I learned that proof from Juris Stepr\={a}ns.}

Here goes.
Fix a large enough regular cardinal $\theta\gg \lambda$, together with a well-ordering $\unlhd_\theta$ of $\mathcal H_\theta$.
Fix a bijection $\psi:(({}^{<\lambda^+}\lambda)\times({}^{<\lambda^+}({}^{\lambda+1}2)))\leftrightarrow \lambda^+$.

For every limit $\alpha<\lambda^+$ with $\sup(\acc(D_\alpha))<\alpha$, let $d_\alpha$ be a cofinal subset of $\alpha$
of order-type $\omega$, consisting of successor ordinals. For $\alpha<\lambda^+$ with $\sup(\acc(D_\alpha))=\alpha$, let $d_\alpha:=\acc(D_\alpha)$.

Fix a limit ordinal $\alpha<\lambda^+$. We would like to determine a function $g_\alpha\in{}^{\le\alpha}({}^{\lambda+1}2)$.
For this, let $\{ \alpha_i\mid i<\otp(d_\alpha)\}$ be the increasing enumeration of $d_\alpha$.
Recursively define a sequence $\langle (p^\alpha_i,f^\alpha_i)\mid i<\otp(d_\alpha)\rangle$, as follows:

$\br$ Let $f_0:=\emptyset$ and $p_0:=\emptyset$.

$\br$ If $i<\otp(d_\alpha)$ and $\langle (p^\alpha_j,f^\alpha_j)\mid j\le i\rangle$ is defined, let
$$\mathcal P^\alpha_{i}:=\{ p\in {}^{<\alpha_{i+1}}\lambda\mid \psi(p,f)\in X_\alpha\cap{\alpha_{i+1}}, p\supseteq p_i, f\supseteq f_i, \dom(p)>\dom(f)=\alpha_i \},$$
$$\mathcal F^\alpha_{i}:=\left\{ f\in  {}^{\alpha_{i}}({}^{\lambda+1}2) \mid \psi(p,f)\in X_\alpha\cap{\alpha_{i+1}}, p=\min_{\unlhd_\theta}\mathcal P_i, f\supseteq f_i \right\},$$
and put
$$(p^\alpha_{i+1},f^\alpha_{i+1}):=\begin{cases}(\min_{\unlhd_\theta}\mathcal P^\alpha_{i},\min_{\unlhd_\theta}\mathcal F^\alpha_{i})&\mathcal P^\alpha_{i}\neq\emptyset\\
(\emptyset,\emptyset),&\text{otherwise}\end{cases}.$$

$\br$ If $i<\otp(d_\alpha)$ is a limit ordinal, and $\langle (p^\alpha_j,f^\alpha_j)\mid j< i\rangle$ is defined, let
$p^\alpha_i:=\bigcup_{j<i}p^\alpha_j$ and $f^\alpha_i:=\bigcup_{j<i}f^\alpha_j$.

This completes the construction of $\langle (p^\alpha_i,f^\alpha_i)\mid i<\otp(d_\alpha)\rangle$.
Put
\begin{itemize}
\item $g_\alpha:=\bigcup\{ f^\alpha_i\mid i<\otp(d_\alpha), \forall j< i(\mathcal P^\alpha_j\neq\emptyset)\}$;
\item $A_\alpha^i:=\{ \beta<\dom(g_\alpha)\mid g_\alpha(\beta)(i)=1, h(\beta)=h(\alpha)\}$ for all $i<\lambda$;
\item $K_\alpha:=\{ \beta<\dom(g_\alpha)\mid g_\alpha(\beta)(\lambda)=1\}$.
\end{itemize}
For every $i<\otp(d_\alpha)$, put $\alpha_i':=\min((K_\alpha\cup\{\alpha_{i+1}\})\setminus\alpha_i+1)$,
and $\alpha_i'':=\min((A_\alpha^i\cup\{\alpha_{i+1}\})\setminus\alpha_i')$. Finally, put:
$$C_\alpha:=\begin{cases}d_\alpha\setminus\dom(g_\alpha),&\dom(g_\alpha)<\alpha\\
\acc(d_\alpha)\cup\{ \alpha_i''\mid i<\otp(d_\alpha), \alpha_i<\alpha_i''<\alpha_{i+1}\},&\text{otherwise}\end{cases}.$$

\begin{lemma}\label{l21} For every limit $\alpha<\lambda^+$:
\begin{enumerate}
\item  $C_\alpha$ is a club in $\alpha$ of order-type $\le\lambda$;
\item if $\beta\in\acc(C_\alpha)$, then $C_\beta=C_\alpha\cap\beta$;
\item if $\otp(C_\alpha)=\lambda$, then $h(\beta)=h(\alpha)$ for all $\beta\in C_\alpha$.
\end{enumerate}
\end{lemma}
\begin{proof} Fix a limit ordinal $\alpha<\lambda^+$.

(1)
$\br$ If $\dom(g_\alpha)<\alpha$, then $C_\alpha=d_\alpha\setminus\dom(g_\alpha)$ is a club in $\alpha$
of order-type $\le\otp(d_\alpha)\le\lambda$. Note that $\acc(C_\alpha)\s \acc(d_\alpha)$.

$\br$ If $\dom(g_\alpha)=\alpha$,
then by $\alpha_i<\alpha_i''\le \alpha_{i+1}$ for all $i<\otp(d_\alpha)$,
we have $\acc(C_\alpha)\s \acc(d_\alpha)$,
and $\otp(C_\alpha)\le \otp(d_\alpha)$. In particular, $C_\alpha$ is a club in $\alpha$ of order-type $\le\lambda$.

(2) Fix $\beta\in\acc(C_\alpha)$.
By $\beta\in\acc(C_\alpha)\s \acc(d_\alpha)$, we have $\otp(d_\alpha)>\omega$ and $d_\alpha=\acc(D_\alpha)$.
In particular, $\beta\in \acc(D_\alpha)$, $X_\beta=X_\alpha\cap\beta$,
 $D_\beta=D_\alpha\cap\beta$, and $d_\beta=\acc(D_\beta)$. Consequently,
 the sequence $\langle (p^\beta_i,\mathcal P^\beta_i,f^\beta_i,\mathcal F^\beta_i)\mid i<\otp(d_\beta)\rangle$
 is an initial segment of the sequence $\langle (p^\alpha_i,\mathcal P^\alpha_i,f^\alpha_i,\mathcal F^\alpha_i)\mid i<\otp(d_\alpha)\rangle$,
 and $g_\beta=g_\alpha\restriction\beta$.

If $\dom(g_\alpha)<\alpha$, then by $\beta\in\acc(C_\alpha)=\acc(d_\alpha\setminus\dom(g_\alpha))$,
we get that $g_\alpha=g_\beta$ and $C_\beta=d_\beta\setminus\dom(g_\beta)=d_\alpha\cap\beta\setminus\dom(g_\alpha)=C_\alpha\cap\beta$.

If $\dom(g_\alpha)=\alpha$, then by $g_\beta=g_\alpha\restriction\beta$, we get that $\{ \beta_i''\mid i<\otp(d_\beta)\}=\{ \alpha_i''\mid i<\otp(d_\alpha)\}\cap\beta$,
and $C_\beta=C_\alpha\cap\beta$.

(3) Clearly, if $\otp(C_\alpha)=\lambda$, then $d_\alpha=\acc(D_\alpha)$.
So  $h(\beta)=h(\alpha)$ for all $\beta\in \acc(C_\alpha)$.
Now,  if $\beta\in C_\alpha\setminus\acc(d_\alpha)$,
then there exists some $i<\otp(d_\alpha)$ such that $\beta=\alpha_i''\in A^i_\alpha\s h^{-1}\{\alpha\}$.
So $h(\beta)=h(\alpha)$.
 \end{proof}

For $i<2$, put:
\begin{itemize}
\item $S_i:=\{\alpha<\lambda^+\mid h(\alpha)=i\}$;
\item $G_i:=\{ \alpha\in S_i\mid \otp(C_\alpha)=\lambda\}$;
\item $E_i:=\{ \{\alpha,\delta\}\in [G_i]^2\mid \min(C_\alpha)>\sup(C_\delta\cap\alpha)\}$.
\end{itemize}

Finally, for $i<2$, let:
\begin{itemize}
\item $V_i:=\{ \chi:\beta\rightarrow\omega\mid \beta\in G_i, \chi\text{ is }E_{(1-i)}\text{-chromatic}\}$;
\item $F_i:=\{ \{ \chi,\chi'\}\in[V_i]^2\mid \{\dom(\chi),\dom(\chi')\}\in E_i, \chi\s \chi'\}$.
\end{itemize}

\begin{lemma} $\chr(V_0\times V_1,F_0* F_1)\le\aleph_0$.
\end{lemma}
\begin{proof} This is where Hajnal's idea \cite{MR815579} come into play.
Define $c:V_0\times V_1\rightarrow\omega$, as follows.
Given $(\chi,\eta)\in V_0\times V_1$, by $G_0\cap G_1=\emptyset$, we have $\dom(\chi)\neq\dom(\eta)$, thus, let
$$c(\chi,\eta):=\begin{cases}2\cdot\chi(\dom(\eta)),&\dom(\chi)>\dom(\eta)\\
2\cdot\eta(\dom(\chi))+1,&\dom(\eta)>\dom(\chi)\end{cases}.$$

Towards a contradiction, suppose that $\{(\chi,\eta),(\chi',\eta')\}\in F_0* F_1$,
while $c(\chi,\eta)=c(\chi',\eta')$, say that they are equal to $n$.

$\br$ If $n$ is is even, we let $\chi^*:=\chi\cup\chi'$. Since $(\chi,\chi')\in F_0$,
we know that $\chi^*$ is $E_1$-chromatic. Since $n$ is even, we have $\dom(\eta),\dom(\eta')\in \chi^*$.
So $\chi^*(\dom(\eta))={n\over 2}=\chi^*(\dom(\eta'))$. But, then, the fact that $\chi^*$ is $E_1$-chromatic entails that $\{\dom(\eta),\dom(\eta')\}\not\in E_1$,
contradicting the hypothesis that $\{\eta,\eta'\}\in F_1$.

$\br$ If $n$ is is odd, we let $\eta^*:=\eta\cup\eta'$. As $(\eta,\eta')\in F_1$,
 $\eta^*$ is $E_0$-chromatic. Since $n$ is odd, we have $\eta^*(\dom(\chi))={n-1\over 2}=\eta^*(\dom(\chi'))$. But, then, the fact that $\eta^*$ is $E_0$-chromatic entails that $\{\dom(\chi),\dom(\chi')\}\not\in E_0$,
contradicting the hypothesis that $\{\chi,\chi'\}\in F_0$.
\end{proof}

\begin{defi}
 For $i<2$ and a limit $\delta<\lambda^+$, write $$C^i_\delta:=\{ \alpha\in C_\delta\cap G_i\mid \min(C_\alpha)>\sup(C_\delta\cap\alpha)\}.$$
\end{defi}
\begin{defi}
For $i<2$ and
 $\gamma<\lambda^+$, we say that a coloring $\chi:\gamma\rightarrow\omega$ is $i$-\emph{suitable} if the two hold:
\begin{itemize}
\item $\chi[C^i_\delta]$ is finite for all $\delta\le\gamma$, and
\item $\chi(\alpha)\neq\chi(\delta)$ for all $\alpha<\delta\le\gamma$ with $\{\alpha,\delta\}\in E_i$.
\end{itemize}
\end{defi}

\begin{lemma}\label{c111}
For every $i<2$, $\beta<\gamma<\lambda^+$
with $\beta\not\in G_i$, and an $i$-suitable coloring $\chi:\beta\rightarrow\omega$,
there exists an $i$-suitable coloring $\chi':\gamma\rightarrow\omega$
extending $\chi$.
\end{lemma}
\begin{proof} By  virtually the same proof of Claim 3.1.3 from \cite{rinot12}, building on Lemma \ref{l21}(2) above.\end{proof}

\begin{lemma}\label{l26} For $i<2$, the notion of forcing $$\mathbb Q_i:=(\{ \chi:\beta\rightarrow\omega\mid \beta\in\lambda^+\setminus G_i, \chi\text{ is }i\text{-suitable}\},\s)$$  is $(\le\lambda)$-distributive.
\end{lemma}
\begin{proof} For concreteness, we work with $\mathbb Q_1$.

Suppose that $\langle \Omega_i\mid i<\lambda\rangle$ is a given sequence of dense open subsets of $\mathbb Q_1$,
$p_0$ is an arbitrary condition, and let us show that there exists $p\in \bigcap_{i<\lambda}\Omega_i$ extending $p_0$.
Let $\langle N_\alpha\mid \alpha<\lambda^+\rangle$
be an increasing and continuous sequence of elementary submodels of $\left(\mathcal H(\theta),\in,\le_\theta\right)$,
each of size $\lambda$, such that $\langle D_\delta\mid\delta<\lambda^+\rangle,\mathbb Q_1,\langle \Omega_i\mid i<\lambda\rangle,p_0\in N_0$,
and $\langle N_\beta\mid \beta\le\alpha\rangle\in N_{\alpha+1}$ for all $\alpha<\lambda^+$.

Put $E:=\{\delta<\lambda^+\mid N_\delta\cap\lambda^+=\delta\}$. By the choice of
$\langle (D_\alpha,X_\alpha)\mid \alpha<\lambda^+\rangle$,
let us pick some $\alpha<\lambda^+$ with $\otp(D_\alpha)=\lambda$ such that $h(\alpha)=0$ and $\acc(D_\alpha)\s E$.

Let $\{ \alpha_i\mid i\le \lambda\}$ denote the increasing enumeration of $\acc(D_\alpha)\cup\{\alpha\}$.
Write $M_i:=N_{\alpha_i}$.
Notice that for all $i<\lambda$, by $\langle N_\beta\mid \beta\le\alpha_i\rangle\in N_{\alpha_i+1}\s M_{i+1}$ and $\langle D_\delta\mid\delta<\lambda^+\rangle\in M_{{i+1}}$,
we have $\langle M_{j}\mid j\le i\rangle\in M_{{i+1}}$.
Also notice that for all $i\le\lambda$, we have $h(\alpha_i)=0$, and $M_i\cap\lambda^+=\alpha_i\in S_0$.
In particular, $\alpha_i\in\lambda^+\setminus G_1$.

We shall recursively define an increasing sequence of conditions $\langle p_i\mid i<\lambda\rangle$ that will satisfy the following for all $i<\lambda$:
\begin{itemize}
\item $p_{i+1}\in \Omega_i$;
\item $\langle p_j\mid j\le i\rangle\in M_{{i+1}}$;
\item $\dom(p_i)=\alpha_i$ whenever $i>0$.
\end{itemize}

By recursion on $i<\lambda$:

$\br$ $p_0$ was already given to us, and indeed $p_0\in M_{1}$.

$\br$ Suppose that $i<\lambda$, and $\langle p_j \mid j\le i\rangle$ has already been defined, and is an element of $M_{i+1}$.
In particular, $p_i\in M_{{i+1}}$.
We claim that the set $\Psi_i:=\{ q\in \Omega_i\mid q\supseteq p_i, \dom(q)=\alpha_{i+1}\}$ is nonempty.
To see this, notice that since $p_i,\Omega_i\in M_{i+1}$, elementarity of $M_{i+1}$,
yields some $p\in \Omega_i\cap M_{i+1}$ extending $p_i$. Then, by $M_{i+1}\cap\lambda^+=\alpha_{i+1}$,
we have $\dom(p)<\alpha_{i+1}$, and then by Lemma \ref{c111}, we infer the existence of a 1-suitable coloring $q$ extending $p$
with $\dom(q)=\alpha_{i+1}$. As $\alpha_{i+1}\in S_0$, $q$ is a legitimate condition, and as $\Omega_i$ is open, we get that $q$ is in $\Omega_i$,
testifying that $\Psi_i$ is nonempty.

Thus, we let $p_{i+1}$ be the $\le_\theta$-least element of $\Psi_i$.
By the fact that $\Psi_i$ is defined from parameters within $M_{i+2}$, and by the canonical choice of $p_{i+1}$, we have
$p_{i+1}\in M_{i+2}$. Altogether, $\langle p_j\mid j\le i+1\rangle\in M_{i+2}$.

$\br$  Suppose that $i<\lambda$ is a nonzero limit ordinal, and $\langle p_j \mid j< i\rangle$ has already been defined
by our canonical process.
Put $p_i:=\bigcup_{j<i}p_j$.
Then $\dom(p_i)=\alpha_i$,
and since $p_i$ is the limit of an increasing chain of 1-suitable colorings, $p_i$ is $E_1$-chromatic,
and $p_i[C^1_\beta]$ is finite for every $\beta<\alpha_{i}$.
Thus, to see that $p_i$ is 1-suitable, we are left with verifying that $p_i[C^1_{\alpha_i}]$ is finite.
As $h(\alpha_{i})=0$, we get from Lemma \ref{l21}(2) and Lemma \ref{l21}(3) that $h(\beta)\neq 1$ for all $\beta\in C_\alpha\supseteq C_{\alpha_i}$,
so $C^1_{\alpha_i}=\emptyset$,
which entails that $p_i[C^1_{\alpha_i}]$ is finite indeed.
Thus, $p_i$ is a legitimate condition.

By the canonical process, and the fact that $\langle M_j\mid j\le i\rangle\in M_{i+1}$,
we have $\langle p_j\mid j<i\rangle\in M_{i+1}$, and hence $p_i=\bigcup_{j<i}p_j\in M_{i+1}$. So $\langle p_j\mid j\le i\rangle\in M_{i+1}$.

This completes the construction.

Put $p:=\bigcup_{i<\lambda}p_i$. Then $p$ is $E_1$-chromatic, and $p[C^1_\beta]$ is finite
for every $\beta<\alpha$. As $\dom(p)=\alpha$ and $C^1_\alpha$ is empty, we get that $p$ is a legitimate condition.
Consequently, $p$ is an element of  $\bigcap_{i<\lambda}\Omega_i$ that extends $p_0$.
\end{proof}

It is clear that $|V_i|\le 2^\lambda=\lambda^+$ for $i<2$, thus we are left with establishing the following.

\begin{lemma}\label{l5} $\chr(V_i,F_i)=\lambda^+$ for every $i<2$.
\end{lemma}
\begin{proof} For concreteness, we prove that $\chr(V_0,F_0)=\lambda^+$.

Towards a contradiction, suppose that $c:V_0\rightarrow\lambda$ is $F_0$-chromatic.
Let $\mathbb G$ be $\mathbb Q_1$-generic over $V$, and work in $V[\mathbb G]$.

Put $\chi^*:=\bigcup\mathbb G$. Since $\mathbb G$ is directed, for every $\alpha,\delta\in\dom(\chi^*)$,
there exists some $\chi\in\mathbb G$ such that $\{\alpha,\delta\}\in\dom(\chi)$,
and hence $\chi^*(\alpha)\neq\chi^*(\delta)$ whenever $\alpha,\delta\in E_1$.
By Lemma \ref{c111}, we also know that $\dom(\chi^*)\ge \gamma$ for all $\gamma<\lambda^+$.
Altogether, $\chi^*:\lambda^+\rightarrow\omega$ is an $E_1$-chromatic coloring,
and so are its initial segments. In particular, we may derive a coloring $c^*:G_0\rightarrow\lambda$ by letting $c^*(\beta):=c(\chi^*\restriction\beta)$
for all $\beta\in G_0$. Since $c$ is $F_0$-chromatic, we then get that $c^*$ is $E_0$-chromatic.
That is, $c^*$ witnesses that $\chr(G_0,E_0)\le\lambda$.

For all $i<\lambda$, put $H_i:=\{\alpha\in G_0\mid c^*(\alpha)=i\}$, and $M_i:=\{\min(C_\alpha)\mid \alpha\in H_i\}$.
Define a function $h_i:\lambda^+\rightarrow\lambda^+$
by letting for all $\tau<\lambda^+$,
$$h_i(\tau):=\begin{cases}\min\{\alpha\in H_i\mid \min(C_\alpha)>\tau\},&\sup(M_i)=\lambda^+\\
\sup(M_i),&\text{otherwise}\end{cases}.$$

Then, for all $i<\lambda$, put
$$A_i:=\begin{cases}\rng(h_i),&\sup(M_i)=\lambda^+\\
        \lambda^+
,&\sup(M_i)<\lambda^+
       \end{cases},$$
and
$$K:=\{\beta<\lambda^+\mid     \forall i<\lambda, h_i[\beta]\s \beta\}.$$

Finally, define a function $g:\lambda^+\rightarrow {}^{\lambda+1}2$  by letting
$g(\alpha)(i)=1$ iff ($i<\lambda$ and $\alpha\in A_i$) or ($i=\lambda$ and $\alpha\in K$).
Note that by Lemma \ref{l26}, we know that any initial segment of $g$ belongs to the ground model.

Work back in $V$.
Let $p_0\in\mathbb Q_1$ be such that $$p_0\forces\name{g}:\check\lambda^+\rightarrow {}^{\check\lambda+1}2,\text{ and }c^*\text{ is }E_0\text{-chromatic}.$$
By possibly extending $p_0$, we may moreover assume that $p_0$ forces that $\{ \alpha<\lambda^+\mid g(\alpha)(i)=1\}$
is unbounded in $\lambda^+$ for all $i\le\lambda$,
and knows about the interaction of $g$ with $c^*$.

As any initial segment of $g$ belongs to $V$, it makes sense to  consider the  set $$Z:=\{ (p,f)\in\mathbb Q_1\times{}^{<\lambda^+}({}^{\lambda+1}2)\mid p_0\s p\forces_{\mathbb Q_1}\name{g}\restriction\dom(f)=\check f\}.$$

Let $\langle N_\alpha\mid \alpha<\lambda^+\rangle$
be an increasing and continuous sequence of elementary submodels of $\left(\mathcal H(\theta),\in,\le_\theta\right)$,
each of size $\lambda$, such that $\langle D_\delta\mid\delta<\lambda^+\rangle,\mathbb Q_1,\psi,\name{g},p_0\in N_0$,
and $\langle N_\beta\mid \beta\le\alpha\rangle\in N_{\alpha+1}$ for all $\alpha<\lambda^+$.

Put $E:=\{\delta<\lambda^+\mid N_\delta\cap\lambda^+=\delta\}$. By the choice of
$\langle (D_\alpha,X_\alpha)\mid \alpha<\lambda^+\rangle$,
let us pick some $\alpha<\lambda^+$ with $\otp(D_\alpha)=\lambda$ such that $h(\alpha)=0$, $X_\alpha=\psi[Z]\cap\alpha$, and $\acc(D_\alpha)\s E$.

Let $\{ \alpha_i\mid i\le \lambda\}$ denote the increasing enumeration of $\acc(D_\alpha)\cup\{\alpha\}$.
Write $M_i:=N_{\alpha_i}$.
Notice that for all $i<\lambda$,
we have $\langle M_{j}\mid j\le i\rangle\in M_{{i+1}}$.
Also, we have $h(\alpha_i)=0$, and $M_i\cap\lambda^+=\alpha_i\in S_0$ for all $i\le\lambda$.

We shall recursively define a sequence pairs $\langle (p_i,f_i)\mid i<\lambda\rangle$ that will satisfy the following for all $i<\lambda$:
\begin{itemize}
\item $p_{i+1}\forces \name{g}\restriction\check \alpha_i=\check f_{i+1}$;
\item $\alpha_i<\dom(p_i)<\alpha_{i+1}$;
\item $\langle p_j\mid j\le i\rangle$ is an increasing sequence of conditions that belongs to $ M_{{i+1}}$.
\end{itemize}

By recursion on $i<\lambda$:

$\br$ $p_0$ was already given to us, and indeed $p_0\in M_{1}$.
Put $f_0:=\emptyset$.

$\br$ Suppose that $i<\lambda$, and $\langle p_j \mid j\le i\rangle$ has already been defined, and is an element of $M_{i+1}$.
In particular, $p_i\in M_{{i+1}}$. By Lemmas \ref{c111} and \ref{l26},
the set $\Psi_i:=\{ q\in\mathbb Q_1\mid q\supseteq p_i, \alpha_i<\dom(q)<\alpha_{i+1}, q\text{ decides }\name{g}\restriction\alpha_i\}$ is nonempty.
Thus, we let $p_{i+1}$ be the $\le_\theta$-least element of $\Psi_i$,
and let $f_{i+1}$ be such that $p_{i+1}\forces \name{g}\restriction\check \alpha_i=\check f_{i+1}$.

By the fact that $\Psi_i$ is defined from parameters within $M_{i+2}$, and by the canonical choice of $p_{i+1}$, we have
$p_{i+1}\in M_{i+2}$. Altogether, $\langle p_j\mid j\le i+1\rangle\in M_{i+2}$.

$\br$  Suppose that $i<\lambda$ is a nonzero limit ordinal, and $\langle (p_j,f_j) \mid j< i\rangle$ has already been defined
by our canonical process.
Put $p_i:=\bigcup_{j<i}p_j$, and $f_i:=\bigcup_{j<i}p_j$.
Then $\dom(p_i)=\alpha_i$,
and since $p_i$ is the limit of an increasing chain of 1-suitable colorings, $p_i$ is chromatic,
and $p_i[C^1_\beta]$ is finite for every $\beta<\alpha_{i}$.
Thus, to see that $p_i$ is 1-suitable, we are left with verifying that $p_i[C^1_{\alpha_i}]$ is finite.
As $h(\alpha_{i})=0$, we get from Lemma \ref{l21} that $h(\beta)\neq 1$ for all $\beta\in C_{\alpha_i}$,
so  $p_i[C^1_{\alpha_i}]=\emptyset$ is finite indeed, and $p_i$ is a legitimate condition.

By the canonical process, and the fact that $\langle M_j\mid j\le i\rangle\in M_{i+1}$,
we have $\langle p_j\mid j<i\rangle\in M_{i+1}$, and hence $p_i=\bigcup_{j<i}p_j\in M_{i+1}$. So $\langle p_j\mid j\le i\rangle\in M_{i+1}$.

This completes the construction.
Put $p:=\bigcup_{i<\lambda}p_i$. Then $p$ is a legitimate condition.

Clearly, $\{ (p_i,f_i)\mid i<\lambda\}\s Z$. Note that for all $i<\lambda$,
by $\mathbb Q_1,p_i,\name{g},\alpha_i,\psi\in M_{i+1}$, we have $\psi(p_i,f_i)\in M_{i+1}$. That is, $\psi(p_i,f_i)\in \psi(Z)\cap \alpha_{i+1}=X_\alpha\cap\alpha_{i+1}$.
It follows that $\langle (p_i,f_i)\mid 0<i<\lambda\rangle=\langle (p_i^\alpha,f_i^\alpha)\mid 0<i<\lambda\rangle$!

So, $p\forces \name{g}\restriction\check \alpha=\check g_\alpha$.
Consequently, $p$ forces that $A_i\cap\alpha=A^i_\alpha$ for all $i<\lambda$,
and $K\cap\alpha=K_\alpha$. Also, since
$p_0$ forces that $\{ \alpha<\lambda^+\mid g(\alpha)(i)=1\}$
is unbounded in $\lambda^+$ for all $i\le\lambda$, we  get that $\sup(K_\alpha\cap\alpha_i)=\sup(A_\alpha^i\cap\alpha_i)=\alpha_i$ and $\alpha_i<\alpha_i''<\alpha_{i+1}$
for all $i<\lambda$. In particular, $\{ \alpha_i''\mid i<\lambda\}\s C_\alpha$,
and  $p\forces \min(C_\alpha)=\alpha_0''\ge\min(K)$.
Let $p^*$ be an extension of $p$ that decides $c^*(\alpha)$, say $p^*\forces c^*(\alpha)=\check i$,
and decides $h_i\restriction\alpha$.

Then $p^*$ forces that $\sup(M_i)=\lambda^+$,
because otherwise $$\sup(M_i)<\min(K)\le\min(C_\alpha),$$ contradicting the fact that
$i=c^*(\alpha)$ entails $\sup(M_i)\ge\min(C_\alpha)$.

The upcoming considerations are all forced by $p^*$.
We have $\alpha_i<\alpha_i'\le\alpha_i''<\alpha_{i+1}$,
with $\alpha_i'\in K$ and $\alpha_i''\in A_i\cap C_\alpha$.
By $\alpha_i''\in A_i$ and $\sup(M_i)=\lambda^+$, we have $\alpha_i''\in\rng(h_i)$.
Fix $\tau<\alpha$ such that $h_i(\tau)=\alpha_i''$. Then $\min(C_{\alpha_i''})>\tau$.
By $h_i[\alpha_i']\s\alpha_i'\le\alpha_i''=h_i(\tau)$, we have $\tau\ge\alpha_i'$,
and hence $\min(C_{\alpha_i''})>\tau\ge\alpha_i'>\sup(C_\alpha\cap\alpha_i'')$.
It follows that $\{\alpha_i'',\alpha\}\in E_0$. Recalling that $\alpha_i''\in\rng(h_i)\s H_i$,
we conclude that $c^*(\alpha_i'')=i=c^*(\alpha)$. So $p^*$ forces that $c^*$ is not an $E_0$-chromatic coloring,
contradicting the fact that $p^*$ extends $p_0$.
\end{proof}

\subsection*{Proof of Corollary 1} Suppose that $\kappa$ is a successor cardinal.
If $\kappa=\aleph_1$, then Hajnal's example \cite{MR815579} apply.
Otherwise, $\kappa=\lambda^+$ for some uncountable cardinal, so we are done by recalling that In G\"odel's constructible universe,
$\sd_\lambda$ holds for every uncountable cardinal $\lambda$.\footnote{In fact,
a minor modification to the proof of the main theorem allows to derive the case $\kappa=\aleph_1$, as well.}

\subsection*{Proof of Corollary 2} Given a positive integer $n$,
pick graphs $\mathcal G_0,\mathcal G_1$ of size and chromatic number $\aleph_n$
such that $\mathcal G_0\times\mathcal G_1$ is countably chromatic.
Since $\gch$ holds in G\"odel's constructible universe, we get that $\aleph_0=\underbrace{\log\cdots\log}_{n\text{ times}}(\aleph_n)$. Hence,
$$\chr(\mathcal G_0\times\mathcal G_1)=\underbrace{\log\cdots\log}_{n\text{ times}}(\min\{\chr(\mathcal G_0),\chr(\mathcal G_1)\}).$$

\section{A generalization}\label{sec3}
The main result of this paper generalizes as follows.
\begin{ThM} If $\lambda$ is an uncountable cardinal, and $\sd_\lambda$ holds,
then for every positive integer $n$, there exist graphs $\langle \mathcal G_i\mid i<n+1\rangle$  of size $\lambda^+$ such that:
\begin{itemize}
\item $\chr(\times_{i\in I}\mathcal G_i)=\lambda^+$ for every $I\in[n+1]^n$;
\item $\chr(\times_{i< n+1}\mathcal G_i)=\aleph_0$.
\end{itemize}
\end{ThM}
\begin{proof} Let $\langle (D_\alpha,X_\alpha)\mid \alpha<\lambda^+\rangle$, $h:\lambda^+\rightarrow\lambda^+$,
and $\langle C_\alpha\mid\alpha<\lambda^+\rangle$ be as in the proof of the previous section.
For all $i<\omega$, put:
\begin{itemize}
\item $S_i:=\{\alpha<\lambda^+\mid h(\alpha)=i\}$;
\item $G_i:=\{ \alpha\in S_i\mid \otp(C_\alpha)=\lambda\}$;
\item $E_i:=\{ \{\alpha,\delta\}\in [G_i]^2\mid \min(C_\alpha)>\sup(C_\delta\cap\alpha)\}$;
\item $C^i_\delta:=\{ \alpha\in C_\delta\cap G_i\mid \min(C_\alpha)>\sup(C_\delta\cap\alpha)\}$ for all $\delta<\lambda^+$.
\end{itemize}

For $i<\omega$ and $\gamma<\lambda^+$, we say that a coloring $\chi:\gamma\rightarrow\omega$ is $i$-\emph{suitable} if the two hold:
\begin{itemize}
\item $\chi[C^i_\delta]$ is finite for all $\delta\le\gamma$, and
\item $\chi(\alpha)\neq\chi(\delta)$ for all $\alpha<\delta\le\gamma$ with $\{\alpha,\delta\}\in E_i$.
\end{itemize}

As in the previous section, for every $i<\omega$ and $\beta<\gamma<\lambda^+$
with $\beta\not\in G_i$, and an $i$-suitable coloring $\chi:\beta\rightarrow\omega$,
there exists an $i$-suitable coloring $\chi':\gamma\rightarrow\omega$
extending $\chi$.

Put $\mathbb Q_i:=(\{ \chi:\beta\rightarrow\omega\mid \beta\in\lambda^+\setminus G_i, \chi\text{ is }i\text{-suitable}\},\s)$.
Then a straight-forward variation of the proof of Lemma \ref{l26} shows that the product forcing $\times_{i\in I}\mathbb Q_i$
is   $(\le\lambda)$-distributive for every $I\in[\omega]^{<\omega}$.
Moreover, for $I\in[\omega]^{<\omega}$, as $\langle G_i\mid i\in I\rangle$ are pairwise disjoint,
 the product forcing $\times_{i\in I}\mathbb Q_i$ is isomorphic to
$$\mathbb Q_I:=\left(\left\{ \chi:\beta\rightarrow\omega\mid \beta<\lambda^+\ \&\ \bigwedge_{i\in I}(\beta\not\in G_i\ \&\ \chi\text{ is }i\text{-suitable})\right\},\s\right).$$

Finally, fix a positive integer $n<\omega$, and  all for $i<n+1$, put:
\begin{itemize}
\item $V_i:=\{ \chi:\beta\rightarrow\omega\mid \beta\in\biguplus\{ G_j\mid j<n+1, j\neq i\}, \chi\text{ is }E_{i}\text{-chromatic}\}$;
\item $F_i:=\{ \{ \chi,\chi'\}\in[V_i]^2\mid \{\dom(\chi),\dom(\chi')\}\in\biguplus_{j<n+1}E_j, \chi\s \chi'\}$;
\item $\mathcal V_i:=(V_i,F_i)$.
\end{itemize}

\begin{claim} $\chr(\mathcal V_0\times\cdots\mathcal V_n)\le\aleph_0$.
\end{claim}
\begin{proof} Define $c:V_0\times\cdots\times V_n\rightarrow[\omega^3]^{<\omega}$
by stipulating:
$$c(\chi_0,\ldots,\chi_n):=\{ (\chi_i(\dom(\chi_j)),i,j)\mid i,j<n+1, \dom(\chi_j)\in G_i\cap\dom(\chi_i)\}.$$

Note that, by definition of $V_i$, $h(\dom(\chi_i))\neq i$ for all $i\le n$.
Let us also point out that $c(\chi_0,\ldots,\chi_n)$ is nonempty. For this, define a sequence $\langle a_i \mid i< n+1\rangle$
by letting $a_0:=\chi_0$, and $a_{j+1}:=\chi_{h(\dom(a_j))}$ for all $j<n$.

If there exists some $j<n$ such that $\dom(a_j)<\dom(a_{j+1})$,
then clearly $$\left(a_{j+1}(\dom(a_j)),h(\dom(a_{j+1})),h(\dom(a_j))\right)\in c(\chi_0,\ldots,\chi_n),$$
and we are done. Otherwise, we have $\dom(a_0)>\dom(a_1)>\cdots>\dom(a_n)$, so put $a_{n+1}:=\chi_{h(\dom(a_n))}$.
Let $i<n$ be such that $a_{n+1}=a_i$. Then $\dom(a_{n+1})=\dom(a_i)>\dom(a_n)$,
and hence $$\left(a_{n+1}(\dom(a_n)),h(\dom(a_{n+1})),h(\dom(a_n))\right)\in c(\chi_0,\ldots,\chi_n).$$

Finally, suppose towards a contradiction that $\{(\chi_0,\ldots,\chi_n),(\chi_0',\ldots,\chi_n')\}\in F_0*\cdots *F_n$,
while $c(\chi_0,\ldots,\chi_n)=c(\chi_0',\ldots,\chi_n')$. Pick $(m,i,j)\in c(\chi_0,\ldots,\chi_n)$.
By $(\chi_i,\chi'_i)\in F_i$, we know that $\chi^*:=\chi_i\cup\chi'_i$ is $E_i$-chromatic.
So, by $\chi^*(\dom(\chi_j))=m=\chi^*(\dom(\chi_j'))$, we get that $\{\dom(\chi_j),\dom(\chi_j')\}\not\in E_i$,
contradicting the fact that $\{\chi_j,\chi_j'\}\in F_j$ and $h(\dom(\chi_j))=i=h(\dom(\chi_j'))$.
\end{proof}

\begin{claim} $\chr(\times_{i\in I}\mathcal V_i)=\lambda^+$ for every $I\in[n+1]^n$.
\end{claim}
\begin{proof} Fix $I\in[n+1]^n$. Let $k<n+1$ be such that $n+1=(I\uplus\{k\})$.

Towards a contradiction, suppose that $c:\times_{i\in I}V_i\rightarrow\lambda$ is $*_{i\in I}F_i$-chromatic.
Let $\mathbb G$ be $\mathbb  Q_I$-generic over $V$, and work in $V[\mathbb G]$.
Put $\chi^*:=\bigcup\mathbb G$. Then $\chi^*:\lambda^+\rightarrow\omega$ is $E_i$-chromatic for all $i\in I$.
Notice that for all $i\in I$ and $\beta\in G_k$, by $i\neq k$, we have $\chi^*\restriction\beta\in V_i$.
Thus, we may derive a coloring $c^*:G_k\rightarrow\lambda$ by letting for all $\beta\in G_k$: $$c^*(\beta):=c\left(\prod_{i\in I}\chi^*\restriction\beta\right).$$
Since $c$ is $*_{i\in I}F_i$-chromatic, we then get that $c^*$ is $E_k$-chromatic.
That is, $c^*$ witnesses that $\chr(G_k,E_k)\le\lambda$.

For concreteness, let us assume that $k=0$.
Define $H_i,M_i,h_i,A_i,K,g$ as in the proof of Lemma \ref{l5}.
Work back in $V$.
Let $p_0\in\mathbb Q_I$ be such that $$p_0\forces\name{g}:\check\lambda^+\rightarrow {}^{\check\lambda+1}2,\text{ and }c^*\text{ is }E_0\text{-chromatic}.$$
By possibly extending $p_0$, we may moreover assume that $p_0$ forces that $\{ \alpha<\lambda^+\mid g(\alpha)(i)=1\}$
is unbounded in $\lambda^+$ for all $i\le\lambda$,
and knows about the interaction of $g$ with $c^*$.

As any initial segment of $g$ belongs to $V$, we shall consider the  set $$Z:=\{ (p,f)\in\mathbb Q_I\times{}^{<\lambda^+}({}^{\lambda+1}2)\mid p_0\s p\forces_{\mathbb Q_I}\name{g}\restriction\dom(f)=\check f\}.$$

Let $\langle N_\alpha\mid \alpha<\lambda^+\rangle$
be an increasing and continuous sequence of elementary submodels of $\left(\mathcal H(\theta),\in,\le_\theta\right)$,
each of size $\lambda$, such that $\langle D_\delta\mid\delta<\lambda^+\rangle,\mathbb Q_I,\psi,\name{g},p_0\in N_0$,
and $\langle N_\beta\mid \beta\le\alpha\rangle\in N_{\alpha+1}$ for all $\alpha<\lambda^+$.

Pick some $\alpha<\lambda^+$ with $\otp(D_\alpha)=\lambda$ such that $h(\alpha)=0$, $X_\alpha=\psi[Z]\cap\alpha$, and $\acc(D_\alpha)\s E:=\{\delta<\lambda^+\mid N_\delta\cap\lambda^+=\delta\}$.

Let $\{ \alpha_i\mid i\le \lambda\}$ denote the increasing enumeration of $\acc(D_\alpha)\cup\{\alpha\}$.
We have $h(\alpha_i)=0$, and $M_i\cap\lambda^+=\alpha_i\in S_0$ for all $i\le\lambda$.
Write $M_i:=N_{\alpha_i}$.

Recursively and $\unlhd_\theta$-canonically define a continuous sequence of pairs $\langle (p_i,f_i)\mid i<\lambda\rangle$ that will satisfy the following for all $i<\lambda$:
\begin{itemize}
\item $p_{i+1}\forces \name{g}\restriction\check \alpha_i=\check f_{i+1}$;
\item $\alpha_i<\dom(p_i)<\alpha_{i+1}$;
\item $\langle p_j\mid j\le i\rangle$ is an increasing sequence of conditions that belongs to $ M_{{i+1}}$.
\end{itemize}

This process is feasible thanks to the fact that $C^j_{\alpha_i}$ is empty  for every limit $i<\lambda$
and  every $j\in I$.\footnote{Recall that $h(\alpha_i)\neq j$ for all $j\in I$ and $i\le \lambda$.}
Then $\langle (p_i,f_i)\mid 0<i<\lambda\rangle=\langle (p_i^\alpha,f_i^\alpha)\mid 0<i<\lambda\rangle$,
and $p:=\bigcup_{i<\lambda}p_i$ is a legitimate condition.
Let $p^*$ be an extension of $p$ that decides $c^*(\alpha)$, say $p^*\forces c^*(\alpha)=\check i$,
and decides $h_i\restriction\alpha$.
Then $p^*\forces \{\alpha_i'',\alpha\}\in E_0 \ \&\ \alpha_i''\in\rng(h_i)\s H_i$.
So $p^*$ forces that $c^*$ is not an $E_0$-chromatic coloring,
contradicting the fact that $p^*$ extends $p_0$.
\end{proof}
\end{proof}

\section*{Concluding Remarks}
$\bullet$ While the above graphs are derived directly from $\sd_\lambda$,
their analysis rely heavily on forcing arguments (e.g., the ability to change the chromatic number of the involved graphs via ``nice'' notions of forcing).
We believe it is worthwhile to see whether these objects may be analyzed without invoking the forcing machinery.

$\bullet$ We also think that it is worth carrying a systematic study of the extent to which graphs may change their chromatic number
by means of ``nice'' notions of forcing. To mention a few results from the forthcoming \cite{rinot15}:
\begin{enumerate}
\item There exists a graph of chromatic number $\aleph_1$,
that may change its chromatic number to $\aleph_0$ by means of a $ccc$ forcing.
If \textsf{CH} holds, then moreover exists a graph of chromatic number $\aleph_1$,
that may change its chromatic number to $\aleph_0$ via an \emph{absolutely $ccc$} notion of forcing.
 \item If $\gch$ holds, then for every regular cardinal $\kappa$,
there exists a graph of size and chromatic number $\kappa^+$,
that may change its chromatic number to $\kappa$ via a $(<\kappa)$-directed closed,
and $\kappa^+$-$cc$ notion of forcing;
 \item If $\gch+\diamondsuit(\omega_1)$ holds, then for every regular cardinal $\kappa$,
there exists a graph of size $\kappa^+$, such that any of its subgraphs of size $\kappa^+$
have chromatic number $\kappa^+$, yet,
the chromatic number of the graph may be made $\kappa$ via a $(<\kappa)$-directed closed,
and $\kappa^+$-$cc$ notion of forcing;
\item If $V=L$, then for every uncoutable cardinal $\mu$ below the first fixed point of the $\aleph$-function,
there exists a graph $\mathcal G$ of size and chromatic number $\mu$,
and a sequence of cofinality-preserving notions of forcing $\langle \mathbb P_\alpha\mid\alpha<\mu\rangle$,
such that for every cardinal $\kappa\le\mu$, $\forces_{\mathbb P_\kappa}\chr(\mathcal G)=\kappa$.
\end{enumerate}

\section*{Acknowledgement}
Part of this work was done while I was a postdoctoral fellow
at the Fields Institute and University of Toronto Mississauga, under the supervision of Ilijas Farah,
Stevo Todorcevic, and William Weiss. I would like to express my deep gratitude to my supervisors and the hosting institutes.

I am grateful to Juris Stepr\={a}ns for illuminating discussions on \cite{sh:64},\cite{sh:98}.

\bibliographystyle{plain}

\end{document}